\documentclass[12pt,reqno]{amsart}
\usepackage[a4paper]{geometry}

\usepackage{amsfonts,amsmath,amsthm,amssymb}

\usepackage{hyperref}

\usepackage{tikz}
\usetikzlibrary{math,calc}

\DeclareMathAlphabet{\mathcal}{U}{BOONDOX-cal}{m}{n}
\SetMathAlphabet{\mathcal}{bold}{U}{BOONDOX-cal}{b}{n}
\DeclareMathAlphabet{\mathbcal} {U}{BOONDOX-cal}{b}{n}

\renewcommand{\leq}{\leqslant}
\renewcommand{\geq}{\geqslant}

\newcommand\vH{\mathcal{v\!H}}

\theoremstyle{plain}
\newtheorem{theorem}{Theorem}
\newtheorem{lemma}{Lemma}

\newtheorem{corollary}{Corollary}

\theoremstyle{definition}

\newtheorem{question}{Question}

\begin{document}

\title[A virt.\@ 2-step nilpotent group with polynomial geodesic growth]{A virtually 2-step nilpotent group \\with polynomial geodesic growth}
\author[A. Bishop]{Alex Bishop}
\address{Section de math\'ematiques,
Universit\'e de Gen\`eve,
rue du Conseil-Général 7-9,
1205 Gen\`eve, Switzerland}
\email{Alexander.Bishop@unige.ch}
\urladdr{alexbishop.github.io}

\author[M. Elder]{Murray Elder}
\address{School of Mathematical and Physical Sciences, University of Technology Sydney, Ultimo NSW 2007 Australia}
\email{murray.elder@uts.edu.au}
\urladdr{sites.google.com/site/melderau}

\thanks{Research supported by Australian Research Council grant DP160100486 and an Australian Government Research Training Program Scholarship.}

\begin{abstract}
A direct consequence of Gromov's theorem is that only virtually nilpotent groups may have polynomial geodesic growth.
However, until now the only known examples of groups with polynomial geodesic were virtually abelian.
In this note we furnish an example of a virtually 2-step nilpotent group having polynomial geodesic growth with respect to a certain finite generating set.
\end{abstract}



\maketitle
\section*{Introduction}

The \emph{geodesic growth} function for a finitely-generated group with respect to a finite (monoid) generating set $S$ counts the number of geodesic words over $S$ with a given upper bound on their length.
Notice that this is bounded from below by the volume growth function which instead counts the number of elements which can be represented by such words.

Bridson, Burillo, \v Suni\'c and the second author~\cite{BBES} investigated groups for which this function is polynomial, building on work of Shapiro~\cite{Pascal}.
In particular, they showed that if a nilpotent group is not virtually cyclic then it has exponential geodesic growth with respect to all finite generating sets.
They also gave an example of a virtually $\mathbb{Z}^2$ group having polynomial geodesic growth with respect to a certain generating set, and provided a sufficient condition for a virtually abelian group to have polynomial geodesic growth.

The first author extended this work by characterising the geodesic growth of virtually abelian groups as \textit{D-finite} (or \emph{holonomic}) for every generating set~\cite{Alex-VirtAbel}.
In particular, this shows that the geodesic growth of a virtually abelian group may only be exponential, or both polynomial and rational.

Here we take the next step, by furnishing the first example of a virtually 2-step nilpotent group having polynomial geodesic growth.
This group contains the integral 3-dimensional Heisenberg group of index $2$. 

Our proof relies on a fact contained in work of Blach\`ere~\cite{blachere2003} that for the integral 3-dimensional Heisenberg group with respect to the (standard) generating set $\{a, a^{-1}, b, b^{-1}\}$,  every element has a geodesic representative which ``switches'' between  $a^{\pm 1}$ letters and $b^{\pm 1}$ letters at most five  times (see Lemma~\ref{lem:heisenberg-geodesic}).
We exploit this somewhat surprising fact to construct a generating set for our example.

Our result opens the door to the  
possibility that some construction of a virtually nilpotent group could have intermediate geodesic growth with respect to some generating set. It also raises the question of whether polynomial geodesic growth is restricted to virtually nilpotent groups of step at most two, or if some construction works for higher steps.

\section{Virtually Heisenberg Group}

Let $G$ be a group with a finite generating set $X$. Then, for each word $w = w_1 w_2 \cdots w_k \in X^*$ we write $|w|_X = k$ for its \emph{word length}, and $\overline{w} \in G$ for the element corresponding to the word $w$. We write $w^R=w_k  \cdots w_2w_1$ for the \emph{reverse} of $w$.
For each element $g \in G$ we write $\ell_X(g) = \min\{|w|_X \colon \overline{w} = g\}$ for the \emph{length} of an element with respect to the generating set $X$.
A word $w \in X^*$ is a \emph{geodesic} if $\ell_X(\overline{w}) = |w|_X$.
Then,
\[\gamma_X(n) = \{ w \in X^* \colon \ell_X(\overline{w}) = |w|_X \leq n \}\] is the \emph{geodesic growth function} of $G$ with respect to $X$.

Consider the discrete Heisenberg group
\[
	\mathcal{H} =
	\left\langle
		a,b
		\mid
		[a,[a,b]] = [b,[a,b]] = 1
	\right\rangle
\]
where $[a,b] = a b a^{-1} b^{-1}$.
We follow the convention of Blach\`ere~\cite{blachere2003} and write $(x,y,z) \in \mathcal{H}$ for the element corresponding to the
 word $[a,b]^z b^y a^x$.
 
We define the virtually Heisenberg group
\begin{equation}\label{eq:presentation-1}
	\vH =
	\left\langle
		a,b,t
	\mid
		[a,[a,b]] = [b,[a,b]] = t^2 = 1,\ a^t = b
	\right\rangle.
\end{equation}
We see that $S = \{a,a^{-1},t\}$ is a generating set for $\vH$ as after a Tietze transform to remove the generator $b$ we may obtain the presentation
\begin{equation}\label{eq:presentation-2}
	\vH =
	\left\langle
		a,t
	\mid
		[a,[a,a^t]] = [a^t,[a,a^t]] = t^2 = 1
	\right\rangle.
\end{equation}
We then observe that the relation $[a^t,[a,a^t]]=1 $ 
 is a consequence of the relations $[a,[a,a^t]] = 1$ and $t^2=1$ and thus we have the presentation
\begin{equation}\label{eq:presentation-3}
	\vH =
	\left\langle
		a,t
	\mid
		[a,[a,a^t]] = t^2 = 1
	\right\rangle.
\end{equation}
To see this, one may verify that 
 \begin{align*}
(atata^{-1}ta^{-1})[a,[a,a^t]]^{-1}(atata^{-1}ta^{-1})^{-1} \\= 
(atata^{-1}ta^{-1})[a,[a,a^t]]^{-1}(atata^{-1}ta^{-1}) \\= tat atata^{-1}ta^{-1}tata^{-1}ta^{-1} \\=[a^t,[a,a^t]].
\end{align*}

We provide a partial view of the Cayley graph of $(\vH,S)$ in Figure~\ref{fig:HeisCG}. Informally, one may think of this group as two copies of $\mathcal H$ glued with a ``twist'' by $t$ edges. Our construction is patently inspired by Cannon's virtually-$\mathbb Z^2$ example with interesting geodesic behaviour \cite{Ereg, NS}.

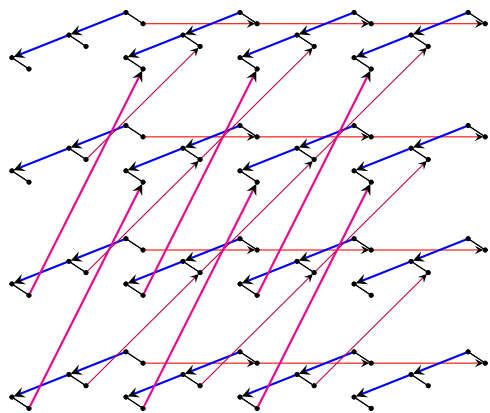
\begin{figure}[ht!]
\centering
\begin{tikzpicture}[scale=1.5,
every node/.style={draw,shape=circle,minimum size=.5mm,inner sep=0pt,outer sep=0pt,fill=black}, >=stealth]

\tikzmath{
	integer \AMax, \BMax, \CMax;
	\AMax = 2;
	\BMax = 3;
	\CMax = 3;
}

\coordinate (x) at (-.5,-.2);
\coordinate (y) at (1,0);
\coordinate (z) at (0,1);
\coordinate (w) at ($-.4*(y)-.5*(x)$);


\foreach \a in {0,...,\AMax} {
	\foreach \b in {0,...,\BMax} {
		\foreach \c in {0,...,\CMax} {
			\node (g\a_\b_\c_0) at ($\a*(x) + \b*(y) + \c*(z)$) {};
			\node (g\a_\b_\c_1) at ($\a*(x) + \b*(y) + \c*(z) + (w)$) {};
			\draw [] (g\a_\b_\c_0) --  (g\a_\b_\c_1); 
		}
	}
}


\tikzset{every path/.style={draw=red}}

\draw [->] (g0_0_0_0) -- (g0_1_0_0);
\draw [->] (g0_1_0_0) -- (g0_2_0_0);
\draw [->] (g0_2_0_0) -- (g0_3_0_0);

\draw [->] (g0_0_1_0) -- (g0_1_1_0);
\draw [->] (g0_1_1_0) -- (g0_2_1_0);
\draw [->] (g0_2_1_0) -- (g0_3_1_0);

\draw [->] (g0_0_2_0) -- (g0_1_2_0);
\draw [->] (g0_1_2_0) -- (g0_2_2_0);
\draw [->] (g0_2_2_0) -- (g0_3_2_0);

\draw [->] (g0_0_3_0) -- (g0_1_3_0);
\draw [->] (g0_1_3_0) -- (g0_2_3_0);
\draw [->] (g0_2_3_0) -- (g0_3_3_0);


\tikzset{every path/.style={draw=blue,thick}}

\foreach \b in {0,...,\BMax} {
	\foreach \c in {0,...,\CMax} {
		\draw [->] (g0_\b_\c_1) -- (g1_\b_\c_1);
		\draw [->] (g1_\b_\c_1) -- (g2_\b_\c_1);
	}
}


\tikzset{every path/.style={draw=purple}}

\draw [->] (g1_0_0_0) -- (g1_1_1_0);
\draw [->] (g1_1_0_0) -- (g1_2_1_0);
\draw [->] (g1_2_0_0) -- (g1_3_1_0);

\draw [->] (g1_0_1_0) -- (g1_1_2_0);
\draw [->] (g1_1_1_0) -- (g1_2_2_0);
\draw [->] (g1_2_1_0) -- (g1_3_2_0);

\draw [->] (g1_0_2_0) -- (g1_1_3_0);
\draw [->] (g1_1_2_0) -- (g1_2_3_0);
\draw [->] (g1_2_2_0) -- (g1_3_3_0);


\tikzset{every path/.style={}}

\foreach \a in {0,...,\AMax} {
	\foreach \b in {0,...,\BMax} {
		\foreach \c in {0,...,\CMax} {
			\node (g\a_\b_\c_0) at ($\a*(x) + \b*(y) + \c*(z)$) {};
			\node (g\a_\b_\c_1) at ($\a*(x) + \b*(y) + \c*(z) + (w)$) {};
			\draw [] (g\a_\b_\c_0) --  (g\a_\b_\c_1); 
		}
	}
}

\tikzset{every path/.style={draw=magenta, thick}}

\draw [->] (g2_0_0_0) -- (g2_1_2_0);
\draw [->] (g2_1_0_0) -- (g2_2_2_0);
\draw [->] (g2_2_0_0) -- (g2_3_2_0);

\draw [->] (g2_0_1_0) -- (g2_1_3_0);
\draw [->] (g2_1_1_0) -- (g2_2_3_0);
\draw [->] (g2_2_1_0) -- (g2_3_3_0);


\end{tikzpicture}
\caption{Cayley graph for $\vH$ with respect to the generating set $S$ where the undirected edges are labelled by $t$ and directed edges labelled by $a$.}\label{fig:HeisCG}
\end{figure}

Our goal is to show that any geodesic of $\vH$ with respect to the generating set $S$ can contain at most $7$ instances of the letter $t$.
From this we are able to place a polynomial upper bound on the geodesic growth function of $\vH$.
To do this, we first study geodesics of the discrete Heisenberg group with respect to the generating set $X = \{a,a^{-1},b,b^{-1}\}$.

Blach\`ere~\cite{blachere2003} provided explicit formulae for the length of elements in $\mathcal{H}$, with respect the generating set $X$, by constructing a geodesic from a given input word.
The following lemma is implicit in Blach\`ere's  work.

\begin{lemma}\label{lem:heisenberg-geodesic}
	Each element $(x,y,z) \in \mathcal{H}$ has a geodesic representative with respect to the generating set $X = \{a,a^{-1},b,b^{-1}\}$ of the form
	\[
		a^{\alpha_1}
		b^{\beta_1}
		a^{\alpha_2}
		b^{\beta_2}
		a^{\alpha_3}
		b^{\beta_3}
	\quad
	\mathrm{or}
	\quad
		b^{\beta_1}
		a^{\alpha_1}
		b^{\beta_2}
		a^{\alpha_2}
		b^{\beta_3}
		a^{\alpha_3}
	\]
	where each $\alpha_i,\beta_j \in \mathbb{Z}$.
\end{lemma}

\begin{proof}
We see that our lemma holds in the case of $(0,0,0) \in \mathcal{H}$ as the empty word $\varepsilon \in S^*$ is such a geodesic.
In the remainder of this proof, we assume that $(x,y,z) \neq (0,0,0)$.
Following Blach\`ere~\cite[p.~22]{blachere2003} we reduce this proof to the case where $x,z \geq 0$ and $-x \leq y \leq x$ as follows.

Let $\tau \colon X^* \to X^*$ be the monoid isomorphism defined such that $\tau(a^k)=b^k$ and $\tau(b^k)=a^k$ for each $k \in \mathbb Z$.
If $w \in X^*$ is a word as described in the lemma statement with $\overline{w} = (x,y,z)$, then $w'=  \tau(w^R)$ is also in the form described in the lemma statement and $\overline{w'} = (y,x,z)$.
Moreover, we see that $\tau(w^R)$ is a geodesic if and only if $w$ is a geodesic.
Defining the monoid isomorphisms $\varphi_a,\varphi_b\colon X^* \to X^*$ by $\varphi_a(a^k) = a^{-k}$, $\varphi_a(b^k) = b^k$, and $\varphi_b(a^k) = a^{k}$, $\varphi_b(b^k) = b^{-k}$ for each $k \in \mathbb{Z}$, we see that if $w \in X^*$ is a geodesic representative for $(x,y,z) \in \mathcal{H}$, then $\varphi_a(w)$, $\varphi_b(w)$ and $\varphi_a(\varphi_b(w))$ are geodesics for $(-x,y,-z)$, $(x,-y,-z)$ and $(-x,-y,z)$, respectively, and each such word is in the form as described in the lemma statement.
From application of the above transformations, we may assume without loss of generality that $x,z \geq 0$ and $-x \leq y \leq x$.

Let $h = (x,y,z) \in \mathcal{H}$, then from \cite[Theorem~2.2]{blachere2003} we have the following formulae for the length $\ell_X(h)$ and (most importantly for us) geodesic representative for $h$.

\begin{itemize}
	\item[I.]
	If $y \geq 0$, then we have the following cases.
	\begin{itemize}
		\item[I.1.]
		If $x < \sqrt{z}$, then $\ell_X(h) = 2\lfloor 2\sqrt{z} \rfloor - x - y$
		and $h$ has a geodesic representative given by $b^{y-y'} S_z a^{x-x'}$ where $x',y'$ are the values given by $\overline{S_z} = (x',y',z)$ (cf.~\cite[p.~32]{blachere2003}), where $S_z$ is as follows.
		
		\begin{itemize}
			\item
			If $z = (n+1)^2$ for some $n \in \mathbb{N}$, then $S_z = a^{n+1} b^{n+1}$;
			\item 
			if there exists a $k \in \mathbb{N}$ with $1 \leq k \leq n$ such that $z = n^2 + k$, then let $S_z = a^k b a^{n-k} b^n$;
			\item
			otherwise, there exists some $k \in \mathbb{N}$ with $1 \leq k \leq n$ such that $z = n^2+n+k$ and we have $S_z = a^k b a^{n+1-k} b^n$.
		\end{itemize}

		\item[I.2.]
		If $x \geq \sqrt{z}$, then we have the following two cases:
		\begin{itemize}
			\item[I.2.1] $xy \geq z$, then $\ell_X(h) = x+y$, otherwise
			\item[I.2.2] $xy \leq z$, then $\ell_X(h) = 2 \lceil z/x \rceil + x - y$;
		\end{itemize}
		and in both cases, the word $b^{y-u-1} a^v b a^{x-v} b^u$ is a geodesic for $h$ where $0 \leq u$, $0 \leq v < x$ and $z = ux+v$ (cf.~\cite[p.~24,\,32,\,33]{blachere2003}).
	\end{itemize}

	\item[II.]
	If $y < 0$, then we have the following cases.
	\begin{itemize}
		\item[II.1.] If $x \leq \sqrt{z - xy}$, then $\ell_X(h) = 2\lceil 2\sqrt{z-xy}\rceil - x + y$.
		Let $n = \lceil \sqrt{z-xy}\rceil-1$.
		Then
		\begin{itemize}
			\item 
			there is either some $k \in \mathbb{N}$ with $1 \leq k \leq n$ such that we have $z-xy = n^2+k$, and $h$ has  $a^{x-n} b^{-n-1} a^k b a^{n-k} b^{n+y}$ as a geodesic representative; or
			\item
			there is some $k \in \mathbb{N}$ with $0 \leq k \leq n$ such that we have $z-xy = (n+1)^2-k$ and $a^{x-n} b^{-k} a^{-1} b^{k-n-1} a^{n+1} b^{n+1+y}$ is a geodesic representative for $h$
			(cf.~\cite[p.~24]{blachere2003}\footnote{Note that in \cite{blachere2003} there is an error in the second case.}).
		\end{itemize}
		\item[II.2.] If $x \geq \sqrt{z - xy}$, then $\ell_X(h) = 2 \lceil z/x \rceil + x - y$ and 
		$h$ has  a geodesic representative of $b^{y-u-1} a^v b a^{x-v} b^u$
		where $u,v \geq 0$, $v < x$ and $z = ux+v$ (cf.~\cite[p.~24,\,33]{blachere2003}).
	\end{itemize}
\end{itemize}
Notice that in each of the above cases, we have our desired result.
\end{proof}

From this lemma, we  have the following result.

\begin{corollary}\label{cor:max-7-t}
	If $w \in S^*$ is a geodesic of $\vH$ with respect to the generating set $S = \{a,a^{-1},t\}$, then $w$  contains at most 7 instances of the letter $t$.
\end{corollary}

\begin{proof}
Let $w \in S^*$ be a word containing $8$ instances of $t$ of the form
\[
	w
	=
	t 
	a^{m_1} t
	a^{n_2} t 
	a^{m_2} t
	a^{n_3} t 
	a^{m_3} t
	a^{n_4} t 
	a^{m_4} t,
\] where $n_i,m_i\in \mathbb Z$, and notice that $\overline{w}$ belongs to the subgroup $\mathcal{H}$.
The Tietze transform given by $b = tat$ which we applied to obtain the presentation  (\ref{eq:presentation-2}) from (\ref{eq:presentation-1}) yields an automorphism $\varphi \colon \vH \to \vH$ given by $\varphi(a) = a$, $\varphi(t) = t$, $\varphi(b) = tat$, and since 
$t^2 = 1$ 
we have $\varphi(b^k) = ta^kt$ for  $k \in \mathbb{Z}$.
Let $X = \{a,a^{-1},b,b^{-1}\}$ be a generating set for the subgroup $\mathcal{H}$.
Then from the word $w \in S^*$ we may construct a word 
\[
	w_2
	=
	b^{m_1}
	a^{n_2}
	b^{m_2}
	a^{n_3}
	b^{m_3}
	a^{n_4}
	b^{m_4}\in X^*
\]
where $\overline{w_2} = \overline{w}$ since $\varphi(w_2) = w$.
Moreover, $|w|_S = |w_2|_X + 8$.

From Lemma~\ref{lem:heisenberg-geodesic}, we know that there is a word $w_3 \in X^*$, with  $\overline{w_3}=\overline{w_2}$ and $|w_3|_X \leq |w_2|_X$, of the form
\[
	w_3
	=
	a^{\alpha_1}
	b^{\beta_1}
	a^{\alpha_2}
	b^{\beta_2}
	a^{\alpha_3}
	b^{\beta_3}
\ \ \text{or}\ \ 
	w_3
	=
	b^{\beta_1}
	a^{\alpha_1}
	b^{\beta_2}
	a^{\alpha_2}
	b^{\beta_3}
	a^{\alpha_3}
\] where $\alpha_i,\beta_i\in\mathbb Z$ (possibly zero).
We then see that $\overline{w}$ can be represented by a word of the form
\[
	w_4
	=
	a^{\alpha_1} t
	a^{\beta_1} t
	a^{\alpha_2} t
	a^{\beta_2} t
	a^{\alpha_3} t
	a^{\beta_3} t
\ \ \text{or}\ \ 
	w_4
	=
	t
	a^{\beta_1} t
	a^{\alpha_1} t
	a^{\beta_2} t
	a^{\alpha_2} t
	a^{\beta_3} t
	a^{\alpha_3}
\]
where \[|w_4|_S \leq  |w_3|_X + 6 \leq  |w_2|_X + 6 < |w_2|_X + 8 = |w|_S.\]

Then $w$ cannot be a geodesic as we have a strictly shorter word $w_4$ that represents the same element.
Thus, a geodesic of $\vH$ with respect to $S = \{a,a^{-1},t\}$ can contain at most $7$ instances of the letter $t$ as we can replace any subword with $8$ instances of $t$ with a strictly shorter word containing at most $7$ instances of $t$.
\end{proof}

From this we obtain the following.

\begin{theorem}\label{thm:main}
	The geodesic growth function of $\vH$ with respect to $S = \{a,a^{-1},t\}$ is bounded from above by a polynomial of degree $8$.
\end{theorem}

\begin{proof}
From Corollary~\ref{cor:max-7-t}, we see that any geodesic of $\vH$, with respect to the generating set $S$, must have the form
\[
	w = a^{m_1} t a^{m_2} t \cdots t a^{m_{k+1}}
\]
where $k \leq 7$ and each $m_i \in \mathbb{Z}$.
Then with $k$ fixed and $r = |w|_S$, we see that there are at most $2^{k+1}$ choices for the sign of $m_1,m_2,\ldots,m_{k+1}$, and  at most $\binom{r}{k}$ choices for the placement of the $t$'s in $w$.
Thus the geodesic growth function $\gamma_S(n)$ has an upper bound given by
\[
	\gamma_S(n)
	\leq
	\sum_{r=0}^n \sum_{k=0}^{7} 2^{k+1} \binom{r}{k}
\]
which give the degree 8 polynomial upper bound.
\end{proof}

\section{Open questions and further work}

The key to our proof of Theorem~\ref{thm:main} is the geodesic representatives given in the work of Blach\`ere. In particular, we exploit the fact that each element of $\mathcal H$ has some geodesic of a particularly special form to prove polynomial geodesic growth for $\vH$.
For this reason, our proof does not immediately appear to generalise to other virtually nilpotent groups (or for that matter to different generating sets of $\vH$).
In light of this, we pose the following question.
\begin{question}[Characterising polynomial geodesic growth]
	For which $k\in \mathbb N$ is there a virtually $k$-step nilpotent group with polynomial geodesic growth with respect to some finite generating set?
\end{question}

It follows from \cite[Theorem~2]{bass1972} that the usual growth rate of a virtually nilpotent group is polynomial of integer degree.
Moreover, from \cite{Alex-VirtAbel} it is known that if a virtually abelian group has polynomial geodesic growth, then it must be of integer degree (since the geodesic growth series is rational in this case).
It is not known if there is a virtually nilpotent group with polynomial geodesic growth of a non-integer degree.
Based on experimental results we conjecture that the geodesic growth rate of $\vH$ with respect to the generating set $S$ can be bounded from above and below by polynomials of degree six (cf.\ the usual growth is polynomial of degree four). 
The corresponding code and first 645 terms of the geodesic growth function are available from \cite{githubcode}.

\begin{question}[Degree of polynomial geodesic growth]
	Is there a group with polynomial geodesic growth of a non-integer degree?
\end{question}

The first author showed that the geodesic growth series for virtually abelian groups is D-finite (holonomic) in the exponential case, and rational in the polynomial case \cite{Alex-VirtAbel}.
It follows from P\'olya-Carlson Theorem~\cite{carlson1921} that a geodesic growth sequence of sub-exponential growth is either rational, or its associated generating function has the unit circle as its natural boundary.
(In particular, such a sequence is either rational or is not D-finite.)
It was shown by Duchin and Shapiro~\cite{MoonS} that the usual growth of $\mathcal H$ is rational for all generating sets. 
Preliminary investigation of the data in \cite{githubcode} leads us to suspect that the geodesic growth sequence for $(\vH,S)$ is not rational, which would mean it is not D-finite.

Finally, we recall a motivating question from \cite{BBES}. 
\begin{question}
	Is there a group with intermediate geodesic growth?
\end{question}
From \cite{BBES,Alex-VirtAbel} it is known that if such a group exists, then it cannot be nilpotent or virtually abelian.

\bibliographystyle{plain}
\bibliography{adm-virtHeis}

\end{document}